\documentclass[12pt]{amsart}

\usepackage[OT2, T1]{fontenc}

\usepackage{amscd}
\usepackage{amsmath}
\usepackage{amssymb}
\usepackage{mathrsfs} 			
\usepackage{units}
\usepackage[all]{xy}

\usepackage{algorithm}
\usepackage{algorithmic}

\def\frk{\frak}               

\def\qq{{\frk q}}

\def\mm{{\frk m}}

\def\Phi{{\frk n}}
\def\Phi{{\frk N}}
\def\opn#1#2{\def#1{\operatorname{#2}}} 
\opn\chara{char} \opn\length{\ell} \opn\pd{pd} \opn\rk{rk}
\opn\projdim{proj\,dim} \opn\injdim{inj\,dim} \opn\rank{rank}
\opn\depth{depth} \opn\sdepth{sdepth} \opn\fdepth{fdepth}
\opn\grade{grade} \opn\height{height} \opn\embdim{emb\,dim}
\opn\codim{codim}  \opn\min{min} \opn\max{max}

\opn\Tr{Tr} \opn\bigrank{big\,rank}
\opn\superheight{superheight}\opn\lcm{lcm}
\opn\trdeg{tr\,deg}
\opn\reg{reg} \opn\lreg{lreg} \opn\ini{in} \opn\lpd{lpd}
\opn\size{size}
\opn\div{div} \opn\Div{Div} \opn\cl{cl} \opn\Cl{Cl}
\opn\Spec{Spec} \opn\Supp{Supp} \opn\supp{supp} \opn\Sing{Sing}
\opn\Ass{Ass} \opn\Min{Min}
\opn\Ann{Ann} \opn\Rad{Rad} \opn\Soc{Soc}
\opn\Im{Im} \opn\Ker{Ker} \opn\Coker{Coker} \opn\Am{Am}
\opn\Hom{Hom} \opn\Tor{Tor} \opn\Ext{Ext} \opn\End{End}
\opn\Aut{Aut} \opn\id{id}  \opn\deg{deg}

\opn\nat{nat}
\opn\pff{pf}
\opn\Pf{Pf} \opn\GL{GL} \opn\SL{SL} \opn\mod{mod} \opn\ord{ord}
\opn\Gin{Gin} \opn\Hilb{Hilb}
\opn\aff{aff} \opn\con{conv} \opn\relint{relint} \opn\st{st}
\opn\lk{lk} \opn\cn{cn} \opn\core{core} \opn\vol{vol}
\opn\link{link} \opn\star{star}
\opn\gr{gr}

\def\pot#1#2{#1[\kern-0.28ex[#2]\kern-0.28ex]}

\opn\dirlim{\underrightarrow{\lim}}
\opn\inivlim{\underleftarrow{\lim}}
\let\iso=\cong

\let\to=\rightarrow

\def\Implies{\ifmmode\Longrightarrow \else
        \unskip${}\Longrightarrow{}$\ignorespaces\fi}
\def\implies{\ifmmode\Rightarrow \else
        \unskip${}\Rightarrow{}$\ignorespaces\fi}
\def\iff{\ifmmode\Longleftrightarrow \else
        \unskip${}\Longleftrightarrow{}$\ignorespaces\fi}

\let\:=\colon
\newtheorem{Theorem}{Theorem}[]
\newtheorem{Lemma}[Theorem]{Lemma}
\newtheorem{Corollary}[Theorem]{Corollary}
\newtheorem{Proposition}[Theorem]{Proposition}

\theoremstyle{definition}

\newtheorem{Remark}[Theorem]{Remark}

\newtheoremstyle{subsection-tweak}
   {11pt}
   {3pt}%
   {}
   {}%
   {\bfseries}
   {}%
   {.5em}
   {\thmnumber{\@{#1}{}\@{#2}.}%
    \thmnote{~{\bfseries#3.}}}    

\newcounter{numberingbase}

\theoremstyle{subsection-tweak}
\newtheorem{bpp}[Theorem]{}
\newtheorem{bppt}[numberingbase]{}
\newcommand{\bbpp}{\begin{bpp}}
\newcommand{\eepp}{\end{bpp}}
\newcommand{\bbppt}{\begin{bppt}}
\newcommand{\eeppt}{\end{bppt}}

\theoremstyle{theorem}

\theoremstyle{definition}

\newcommand{\val}{\mathrm{val}}		

\newcommand{\sU}{\mathscr{U}}

\DeclareMathOperator{\Ir}{Irr}
\DeclareMathOperator{\ch}{char}
\DeclareMathOperator{\card}{card}
\newcommand{\wt}{\widetilde}

\let\epsilon\varepsilon
\let\phi=\varphi
\def\qed{\ifhmode\textqed\fi
      \ifmmode\ifinner\quad\qedsymbol\else\dispqed\fi\fi}
\def\textqed{\unskip\nobreak\penalty50
       \hskip2em\hbox{}\nobreak\hfil\qedsymbol
       \parfillskip=0pt \finalhyphendemerits=0}
\def\dispqed{\rlap{\qquad\qedsymbol}}

\opn\dis{dis}
\def\pnt{{\raise0.5mm\hbox{\large\bf.}}}

\opn\Lex{Lex}

\begin{document}

\title{Extension of valuation rings as limits of complete intersection algebras.}

\author{ Dorin Popescu}

\dedicatory{In the memory of Ionel Bucur (1930-1976) and Nicolae Radu (1931-2001)}

\address{Simion Stoilow Institute of Mathematics of the Romanian Academy,
Research unit 5, P.O. Box 1-764, Bucharest 014700, Romania,}

\address{University of Bucharest, Faculty of Mathematics and Computer Science
Str. Academiei 14, Bucharest 1, RO-010014, Romania,}

\address{ Email: {\sf dorin.m.popescu@gmail.com}}

\begin{abstract} We give sufficient conditions for an extension  of valuation rings  to be a filtered colimit of   complete intersection algebras. 


\end{abstract}

\maketitle

\section*{Introduction}

Following the Zariski Uniformization Theorem \cite{Z} we state in  \cite[Theorem 2]{P1}  the following result.

\begin{Theorem} \label{T0} Let $V\subset V'$ be an extension of valuation rings containing $\bf Q$, $\Gamma\subset \Gamma'$ the value group extension of $V\subset V'$  and $\val:\Gamma'\to K'^{*}$ the valuation of $V'$. Then $V'$ is a filtered colimit of smooth $V$-algebras  if and only if  the following statements hold
\begin{enumerate}

\item for each $\qq\in \Spec V$ the ideal $\qq V'$ is prime,

\item  for any prime ideals $\qq_1,\qq_2\in \Spec V$ such that $\qq_1\subset \qq_2$ and $\height(\qq_2/\qq_1)=1$  and any $x'\in \qq_2V'\setminus \qq_1'$ there exists $x\in V$ such that $\val(x')=\val(x)$,  where $\qq_1'\in \Spec V'$ is the prime ideal corresponding to the maximal ideal of $V_{\qq_1}\otimes_V V'$, that is the maximal prime ideal of $V'$ lying on $\qq_1$.
\end{enumerate}
\end{Theorem}

An {\em immediate} extension of valuation rings is an extension inducing trivial extensions on both the residue field and the   value group.

If the characteristic of the residue field of $V$ is positive then  an immediate extension $V'$ of $V$ might not be a filtered colimit of smooth $V$-algebras  as shows for example \cite[Example 3.13]{Po} (see also \cite[Remark 6.10]{Po})
 inspired by \cite[Sect 9, No 57]{O}.

 For immediate extensions in positive characteristic we have  the following result.
 
\begin{Theorem}(\cite[Theorem 4]{P3}, \cite[Remark 8]{P''})\label{T1}
Let $V\subset V'$ be an  immediate extension of valuation rings with the residue field of positive characteristic and $K\subset K'$ its fraction field extension. If $K'=K(x)$ for some algebraically independent system of elements $x$ over $K$ 
then $V'$ is a filtered union of its smooth $V$-subalgebras. 
\end{Theorem}

Since such theorem in positive characteristic fails in general  we proved a weaker result (in this paper we understand that $p=\ch(k)$ could be also $0$).

\begin{Theorem}(\cite[Theorem 6]{P3} and \cite[Proposition 8]{P2}, especially the arxiv version)\label{T2} Let  $ V'$ be an  immediate extension  of a valuation ring $V$.  Then $V'$ is a filtered
 union of its complete intersection $V$-subalgebras.
\end{Theorem}
A {\em complete intersection}   $V$-algebra   is a  $V$-algebra of type $C/(P)$, where $C$ is  a polynomial $V$-algebra and $P$ is a regular system of elements of $C$.  If $V$ is a filtered colimit of complete intersection algebras we  write shortly {\em ind-ci}. If $V$ is a filtered union of its complete intersection subalgebras we  write shortly {\em injectively ind-ci}.
 
 It is also important  to study a theorem as above for non-immediate extensions. In pure characteristic $0$ this is done in Theorem \ref{T0}.  In pure characteristic $p>0$ such results appeared  in \cite[Theorem 6.2]{KT} and in \cite[Proposition A 10]{KST} following some ideas from \cite[Proposition 6.3.13]{GR}.
 
 Meanwhile in \cite{P''} we got the following theorem similar to \cite[Theorem 9]{P'}.
 
 \begin{Theorem}(\cite[Theorem 22]{P''}) \label{T3} Let $V$ be a Henselian mixed characteristic valuation ring, $k$ its residue field, $p= \ch(k)$, $\val$ its valuation and $\Gamma$ its value group. Assume that  $p>0$ and  $\Gamma /{ \bf Z} \val(p)$ has no $p$-torsion. Then  $V$ is  ind-ci  over ${\bf Z}_{(p)}$.
\end{Theorem}
 
 An abelian group $(G,+)$ has {\em no $p$-torsion} if no nonzero element of it, could be  killed  by powers of $p$.

Thus the above theorem studies a non-immediate extension ${\bf Z}_{(p)}\subset V$ of a mixed  characteristic valuation rings. The goal of this paper is to extend it to general non-immediate extensions.

\begin{Theorem} \label{T4} Let  $ V\subset V'$ be an   extension  of  valuation rings,   $\Gamma\subset \Gamma'$ its value group extension, $k\subset k'$ its residue field extension,  and $p=\ch(k)$.
  Assume that 
   $V'$ is Henselian,
 $k'/k$ is separable,
 and   $\Gamma'/\Gamma$ has no $p$-torsion, if $p>0$.
  Then $V'$ is   ind-ci over $V$.  
\end{Theorem}
The proof of the above theorem goes  using Theorem \ref{t9} and its proof follows mainly the proof of Theorem \ref{T3} given in \cite{P''}. 

We owe thanks to Rodica Dinu for some useful comments on our paper. Also we owe thanks to some anonymous referee who found a gap in Lemma \ref{g} and  had some interesting remarks on  our paper especially on the proof of Theorem \ref{t9}.

\vskip 0.3 cm
\section{Extensions of valuation rings with trivial residue field extensions and torsion value group extensions}

\begin{Lemma} \label{gen}
Let $V\subset V'$ be an extension of valuation rings with  trivial residue field extension,  $K\subset K'$ their fraction field extension, $\Gamma\subset \Gamma'$ their value group extension and $x$ an element of $V'$. Let  $\qq\in \Spec V$ correspond to  the minimal prime ideal of $V'/xV'$. Let $\qq'\in \Spec V'$  correspond to the maximal ideal of $V'[1/x]$. Denote $\gamma=\val(x)$.

Assume  that 

\begin{enumerate}
\item $K'=K(x)$  and $[K':K]=m$.

\item $\Ir(x,K)$  is a polynomial over $V$.
\item  $\Gamma'/\Gamma={\bf Z}\gamma\cong  {\bf Z}/m{\bf Z}$. 

\item $W=(V/\qq'\cap V)_{\qq\cap V}$ is not a DVR.
\item $W$ is a filtered union of its smooth $V$-subalgebras.
\end{enumerate}
Then $V'$ is  injectively ind-ci over $V$.
\end{Lemma}

\begin{proof} 
 An element of $V'$ has the form $f(x)/t$ for some polynomial $f\in V[X]$ of degree $<m$ and $t\in V\setminus \{0\}$, let us say $f=\sum_{i=0}^{m-1}a_iX^i$.
If $a_i,a_j\not =0$, $0\leq i<j<m$ then $\val(a_ix^i)\not =\val(a_jx^j)$, because otherwise we get $\val(a_i)+i\gamma=\val(a_j)+j\gamma$ and so $(j-i)\gamma\in \Gamma$ which is false. Suppose that $\val(a_ix^i)=\min_e\val(a_ex^e)$. It follows that  
$f(x)\in Va_ix^i(1+\mm'\cap V[x])$, $\mm'$ being the maximal ideal of $V'$ and  $\val(f(x)/t)=\val(a_ix^i/t)\geq 0$. When $\val(a_i)\geq \val(t)$ we get $f(x)/t\in V[x]_{\mm'\cap V[x]}$. If $\val(a_i)<\val(t)$ we see that $f(x)/t\in V[x^i/t']_{\mm'\cap V[x^i/t']}$ for  $t'=t/a_i$.
 Actually, $\val(t')<\val(x^i)$ because $i\gamma\not \in \Gamma$. Let $C$ be the $V$-subalgebra of $V'$ generated by all elements of the form $x^i/t$, $0\leq i<m$, $t\in V\setminus \{0\}$ with 
$\val(x^i)> \val(t)$. Then $V'=C_{\mm'\cap C}$.

We show that $C$ is a filtered union of its complete intersection $V$-subalgebras.  Note that $\qq'$ is the biggest prime ideal of $V'$ which does not contain $x$. So  height$(\qq/\qq')=1$ and $\dim W'=1$, $W'=( V'/\qq'V')_{\qq}$. We have $\qq\cap V\not =\qq'\cap V$ because $V'$ is a localization of the integral closure of $V$ in $V'$. Note that $\dim W=1$. Then we may assume that the value group $\Gamma_1$ of  $W'$ is  contained in $\bf R$.

Thus the value group of $W$, a subgroup of $\Gamma_1$,  is dense in $\bf R$ because $W$ is not a DVR.
Let   ${\mathcal T}^-$ be a cofinal subset of $\{\tau\in \Gamma: 0\leq \tau <\gamma\}$.   We may  choose  ${\mathcal T}^-=\{ \val(t_s): s\in {\bf N}\}$ for some $t_s\in V$, with $(\val(t_s))$  an increasing sequence and $\gamma$  its limit in $\bf R$. Then $C=\cup_{s\in {\bf N}} B_s$, $B_s=V[x/t_s]$. Indeed, if $x^r/t'\in C$ for some $r\in {\bf N}$, $t'\in V\setminus\{0\}$ then choose $s\in {\bf N}$ such that $\gamma>\val(t_s)>(1/r)\val(t')$ and we have $x^r/t'\in B_s$ because  
$$x^r/t'=(x/t_s)^r (t_s^r/t')\in V[(x/t_s)^r]\subset V[x/t_s].$$

Clearly, $B_s\subset B_{s'}$ for $s<s'$.  It is enough to see that $B_s$ is ind-ci over $V'$. But $V[z]$ is  ind-ci over $V$ for any $z\in V'\setminus V$. Indeed, $z$ is a root of  an irreducible polynomial $h\in K[Z]$, which could be considered in $V[Z]$ after multiplication with a nonzero element of $V$. Moreover, we may choose $h$ to be primitive. Then $h V[Z]$ is a prime ideal of $V[Z]$. Indeed, if $g_1g_2\in h V[Z]$ for some  $g_1,g_2\in  V[Z]$ then let us say $g_1\in h K[Z]$, that is $cg_1\in hV[Z]$  for some   $c\in V\setminus \{0\}$. It follows that $g_1\in h V[Z]$ because $h$ is primitive.
 So $V[Z]/(h)\cong V[z]$, the isomorphism being given by $Z\mapsto z$.      
\hfill\ \end{proof}

\begin{Proposition} \label{pr}  Let  $ V\subset V'$ be an   extension  of  valuation rings  with 
trivial residue field extension, $k$ the residue field of $V$, $p=\ch(k)$,  $\Gamma \subset  \Gamma'$ their value group extension, $K\subset K'$ their fraction field extension with $[K':K]=m$ and  $x\in K'\setminus K$ with $K'=K(x)$ such that $\Ir(x,K)$ is a polynomial over $V$.   Assume  that $V'$ is Henselian and $\Gamma'/\Gamma={\bf Z}\gamma\cong  {\bf Z}/m{\bf Z}$, $\gamma=\val(x)$ for some $m\not \in p{\bf Z}$. Then $V'$ is 
 ind-ci over $V$.
\end{Proposition}
\begin{proof} We have $m\gamma\in \Gamma$ and so $m\val(x)=\val(y) $ for some $y\in V$. Then $x^m=yu$ for some unit $u\in V'$. Take an unit element $t\in V$ which induces the same element with $u$ in $k$. Changing $y,u$ by $yt,u/t$ we may assume that $\val(u-1)>0$.  Then the equation $Z^m-u=0$ has $1$ as a solution  modulo the maximal ideal $\mm'$ of $V'$ and so it has one $z$ in $V'$ using the Henselian property ($m\not \in p{\bf Z}$).
 Dividing $x$ by $z$ we may suppose that $u=1$. Set $V''=V'\cap K(x)$ and let $\Gamma''$ be the value group of $V''$. Then $\Gamma''=\Gamma'$,  Irr$(x,K)=Y^m-y$ and  $[K'':K]=m$. Let  $\qq$ be the minimal prime over ideal of $xV'$. Let $\qq'\in \Spec V'$ be the prime ideal corresponding to the maximal ideal of the fraction ring of $V'$ with respect to the multiplicative system  generated by $x$. By the above lemma we see that $V''$ is a injectively ind-ci over $V$ when $W=(V/\qq'\cap V)_{\qq\cap V}$ is not a DVR. 
 
 Now suppose that $W$ is a DVR. Then $W'=(V''/\qq'\cap V'')_{\qq\cap V''}$ is a DVR too by 
 \cite[Ch. VI, Corollary 3 of (8.1)]{Bou}. In this case we can assume that $x$ defines a local parameter in $W'$ and $y$ defines a local parameter in $W$.
 As in the proof of Lemma \ref{gen} $V''= C_{\mm'\cap C}$. Let $x^r/t$ be with $t\in V$ and $\val(t)<\val(x^r)$. Thus $x^r\in tV'$ and so $t\not \in \qq'$. If $t\in xV'$ then $t\in yV$ and we may simplify with  $y=x^m$ several times such that finally $t\not \in yV$ and so $t\not \in xV'$. Hence $t\not \in \qq$ and $\val(t)<\val(x)$. It follows that $x^r/t=x^{r-1}(x/t)\in V[x/t]$. Taking ${\mathcal T}^-= \val(\mm\setminus (\qq\cap V))$, $\mm$ being the maximal ideal of $V$, we see as in Lemma \ref{gen} that $C$ is a filtered union   of some $V[x/t]$, $t\not \in \qq\cap V$, which are complete intersection $V$-subalgebras  of $V''$.
 
 Therefore, in both cases $V''$ is  injectively ind-ci over  $V$.
 But the extension $V''\subset V'$ is immediate and it is enough to apply Theorem \ref{T2}.
\hfill\ \end{proof}

\begin{Theorem}\label{t8}   Let  $ V\subset V'$ be an   extension  of  valuation rings  with 
trivial residue field extension, $k$ the residue field of $V$, $p=\ch(k)$,  $\Gamma \subset  \Gamma'$ their value group extension, and $K\subset K'$ their fraction field extension.  Assume  that $V'$ is Henselian and $\Gamma'/\Gamma$ is a finitely generated torsion group but  has no $p$-torsion, if $p>0$. Then $V'$ is 
  ind-ci over $V$.
\end{Theorem}

\begin{proof} 
Let $\gamma\in \Gamma'\setminus \Gamma$ be such that  $m\gamma\in \Gamma$ for some $m\in {\bf N}$,  which is certainly not in $p{\bf Z}$. We may choose a minimal such $m$.  We may find $x\in K'$ such that $\val(x)=\gamma$. As in Proposition \ref{pr}  $x$  could be taken to be a root of a polynomial over $K$ and so $x$ is algebraic over $K$. 

Moreover,
we may assume that  $g=\Ir(x,K)$ is a polynomial over $V$. 
Indeed, multiplying $g$ by some element of $V$ we get an equality $\sum_{i=0}^m a_ix^i=0 $ for some $a_i\in V$. Then $\sum_{i=0}^m a_m^{i-1}a_{m-i}(a_mx)^{m-i}=0$ and so $x'=a_mx$ is a root of the irreducible monic polynomial $g'=\sum_{i=0}^m a_m^{i-1}a_{m-i}X^{m-i}\in V[X]$. Note that $x$ and $x'$ define by value the same element in $\Gamma'/\Gamma$ and so for $\gamma'=\val(x')$ we get still $\Gamma'/\Gamma\iso {\bf Z}\gamma'$

As in Proposition \ref{pr} we may change $x$ such that   $y=x^m\in K$. 
Then the value group of $V''=V'\cap K(x)$ is ${\bf Z}\gamma\cong {\bf Z}/m{\bf Z}$. 
Thus $V''$ is  ind-ci over $V$ by Proposition \ref{pr}. 

Changing $V$ by $ V''$ we see that the torsion part of $\Gamma'/\Gamma$ becomes smaller. As $\Gamma'/\Gamma$ is finitely generated we reduce step by step to the case when $\Gamma'=\Gamma$. 

Therefore,  the extension $V\subset V'$ is immediate and we may apply Theorem \ref{T2}, which is enough. Note that a complete intersection algebra  over a complete intersection $V$-algebra is still a complete intersection $V$-algebra  by \cite[Lemma 6]{P2}
\hfill\ \end{proof}

The following result follows immediately from the above theorem.

\begin{Corollary} \label{cor}
 Let  $ V\subset V'$ be an   extension  of  valuation rings  with 
trivial residue field extension,  $k$ the residue field of $V$ and  $\Gamma\subset \Gamma'$ their value group extension.   
 Assume $k\supset {\bf Q}$,  $V'$ is Henselian and  $\Gamma'/\Gamma$ is  a finitely generated torsion group.  Then $V'$ is  
  ind-ci over  $V$.
\end{Corollary}

\begin{Remark} \label{r}  If $V'$ is Henselian, $\dim V'=\dim V=1$, $k\supset {\bf Q}$, $\Gamma'$ is finitely generated  and  the extension $\Gamma\subset \Gamma'$ is not trivial then it is known that $V'$ might  not  be a filtered direct limit of smooth $V$-algebras (see \cite[Theorem 2]{P}). The above corollary says that in this case $V'$ is   ind-ci over $V$.
\end{Remark}  
\vskip 0.3 cm

\section{Extensions of valuation rings with trivial residue field extensions}

\begin{Lemma} \label{g} 
Let $V\subset V'$ be an extension of valuation rings with  trivial residue field extension,  $K\subset K'$ their fraction fields and $\Gamma\subset \Gamma'$ their value group extension. 
Let $x\in V'$, $\qq\in \Spec V'$ correspond to  the minimal prime ideal of $V'/xV'$. Let $\qq'\in \Spec V'$  correspond to the maximal ideal of $V'[1/x]$.
Assume  that 
\begin{enumerate}
\item
$W=(V/\qq'\cap V)_{\qq\cap V}$ is not a DVR,

\item $K'=K(x)$ and 

\item $\Gamma'/\Gamma={\bf Z}\gamma\cong  {\bf Z}$, the isomorphism being given by $\gamma=\val(x)\mapsto 1$. 
\end{enumerate}
 Then $V'$ is  injectively ind-ci
over $V$.
\end{Lemma}

\begin{proof} By (i) $W$ is either a field, or a one dimensional valuation ring which is not Noetherian. Note that $x$ is transcendental over $K$ because $\val(x)$ induces an element without torsion in $\Gamma'/\Gamma$ (see \cite[Theorem 1, in VI (10.3)]{Bou}). Also note that ${\bf N}\gamma \cap \Gamma=0$.

 Let $\Gamma_{<\gamma}=\{\tau\in \Gamma: 0< \tau<\gamma\}$ and  $\Gamma_{>\gamma}=\{\tau\in \Gamma: \tau> \gamma\}$. 
Choose $a_{\tau}\in V$ such that $\val(a_{\tau})=\tau$ for $\tau\in \Gamma$. Set $y_{\tau}=x/a_{\tau}\in V'$ for $\tau\in \Gamma_{<\gamma}$, $z_{\nu}=a_{\nu}/x$ for $\nu\in \Gamma_{>\gamma}$  and $B_{\tau,\nu}=V[y_{\tau},z_{\nu}]\cong V[Y_{\tau},Z_{\nu}]/(Y_{\tau}Z_{\nu}-a_{\nu}/a_{\tau})$. If $\tau<\tau'\in \Gamma_{<\gamma}$ then $a_{\tau'}=a_{\tau}c$ for some $c\in V$ and $y_{\tau}=y_{\tau'}c.$ If $\nu'<\nu$ in    $\Gamma_{>\gamma}$ then $a_{\nu}=a_{\nu'}t$ for some $t\in V$ and 
$z_{\nu}=z_{\nu'}t.$ Thus $B_{\tau,\nu}\subset B_{\tau',\nu'}$ if  $\tau<\tau' $ and $\nu'<\nu$ and so $A=\cup_{\tau,\nu}B_{\tau,\nu}$ is a filtered union of complete intersection $V$-subalgebras when $\tau$ is increasing in  $\Gamma_{<\gamma}$ and $\nu$ is decreasing in $\Gamma_{>\gamma}$. If $\Gamma_{>\gamma}=\emptyset$ then we set $A=\cup_{\tau} B_{\tau}$ for $B_{\tau}=V[y_{\tau}]$. If $\Gamma_{<\gamma}=\emptyset$ then we set $A=\cup_{\nu}B_{\nu}$ for $B_{\nu}=V[z_{\nu}]$.
It is enough to show that $V'$ is a localization of $A$.

Fix $f=\sum_{i=0}^r c_iX^i\in V[X]$ for some $c_i\in V$. Then the nonzero elements $c_ix^i$ have different values because $\gamma$ has no torsion in $\Gamma'/\Gamma$.
Indeed, if $\val(c_ix^i)=\val(c_jx^j)$ for $i>j$  then $(i-j)\gamma=\val(c_j)-\val(c_i)\in \Gamma$, which is false.
 Assume  $\val(f(x))=\val(c_sx^s)$ for some $s$, and set $d=c_sx^s\in V[x]$. We have $f(x)=dv$ for the unit $v=1+\sum_{i=1, i\not =s}^r (c_ix^i/c_sx^s)  $ in $V'$. We see that $(c_i/c_s)x^{i-s}\in \mm'$ for $i\not =s$, $\mm'$ being the maximal ideal of $V'$.
Moreover  all  $(c_i/c_s)x^{i-s}$ belongs to  $A$. Indeed, if $i>s$ then from $\val((c_i/c_s)x^{i-s})>0$ we get $(i-s)\val(x)>-\val(c_i/c_s)$. When $c_s/c_i\in q'\cap V$ then clearly we get $\val(x)>\val(c_s/c_i)$. If $c_s/c_i\not \in q'\cap V$
then thinking the value group of $W$ contained in $\bf R$  
we get $\val(x)>\val(c_s/c_i)/(i-s)\in {\bf R}$.
Since $W$ is not a DVR, we may  choose $\tau$ in $\Gamma_{<\gamma}$ with $\tau\geq -(\val(c_i/c_s)$ respectively $ \tau \geq -(\val(c_i/c_s)/(i-s))$. In both cases it follows that $(c_i/c_s)x^{i-s}=((c_i/c_s)a_{\tau}^{i-s})(x/a_{\tau})^{i-s}\in V[x/a_{\tau}]\subset B_{\tau,\nu}$ for any $\nu$.  Certainly if $W$ is a field  (that is $\qq\cap V=\qq'\cap V$) then we may choose even $\tau=-\val(c_i/c_s)$. 

 If $i<s$ then we get $\val(c_i/c_s)>(s-i)\val(x)$. As above we  choose $\nu$ in $\Gamma_{>\gamma}$ with  $ \nu \leq \val(c_i/c_s)$ if $c_i/c_s\not \in q\cap V$, respectively 
  $ \nu \leq (\val(c_i/c_s)/(i-s))\in {\bf R}$ if $c_i/c_s\in q\cap V$.
 It follows that  $(c_i/c_s)x^{i-s}=((c_i/c_s)/a_{\nu}^{s-i})(a_{\nu}/x)^{s-i}\in V[(a_{\nu}/x)]
\subset B_{\tau,\nu}$ for any  $\tau$. 

Therefore $v$ is a unit in $A_{\mm'\cap A}$.  If $g_1/g_2$ is a rational function from $V'$, $g_i\in V[x]$ then there exists $d_i,v_i$, $i=1,2$ with $d_i\in V'$ and $v_i$ units in $A_{\mm'\cap A}$ such that \\
$g_i=d_iv_i$. As above $d_1/d_2\in A$. Therefore, $g_1/g_2=(d_1/d_2)v_1v_2^{-1}\in A_{\mm'\cap A}$  and so $A_{\mm'\cap A}=V'$.
\hfill\ \end{proof}

We need the following lemma (see \cite[2.2]{Ell}, or \cite[4.6.1]{Po1}, or \cite[6.1.30]{GR}).

\begin{Lemma} \label{EPGR-lem}
For a totally ordered abelian group $\Gamma$, the submonoid $\Gamma_{\geq 0} \subset \Gamma$ of nonnegative elements is a filtered  union of its finite free submonoids isomorphic to ${\bf Z}_{\geq 0}^r$,  where $r \in { \bf Z}_{\geq 0}$ need not be constant. 
\end{Lemma}

\begin{Proposition} \label{p9} 
Let $V\subset V'$ be an extension of valuation rings with  trivial residue field extension, $\mm,\mm'$ their maximal ideals, $K\subset K'$ their fraction fields and $\Gamma\subset \Gamma'$ their value group extension. 
Let $x\in V'$, $\qq\in \Spec V'$ correspond to  the minimal prime ideal of $V'/xV'$. Let $\qq'\in \Spec V'$  correspond to the maximal ideal of $V'[1/x]$.
Assume  that   $W=(V/\qq'\cap V)_{\qq\cap V}$ is  a DVR, $K'=K(x)$ and $\Gamma'/\Gamma={\bf Z}\gamma\cong  {\bf Z}$, the isomorphism being given by $\gamma=\val(x)\mapsto 1$.  Then $V'$ is  injectively ind-ci 
over $V$.
\end{Proposition}
\begin{proof}  Let $G={\bf Z}\val(\pi) +{\bf Z}\gamma\subset \Gamma'$. 
Use Lemma \ref{EPGR-lem} to find a countable sequence $M_0 \subset M_1 \subset \dots$ of submonoids of $G_{\geq 0}$ with $M_i \simeq {\bf Z}^2_{\geq 0}$ for each $i$ and $G_{\geq 0}=\cup_i M_i$ . We fix a ${\bf Z}_{\geq 0}$-basis $\nu_{i0}, \nu_{i 1}$ of $M_i$ with $(\nu_{00}, \nu_{01}) = (\gamma, \val(\pi))$, so that the elements $\nu_{i0}, \nu_{i 1}$ are $\bf Z$-linearly independent in $G$, and we express them in terms of the fixed $\bf Z$-basis:
\[
\nu_{ij} =  d_{ij 0 } \gamma +  d_{ij 1 } \val(\pi) \ \ \mbox{for unique} \ \ d_{ij 0 },  d_{ij 1 } \in {\bf Z} \ \ \mbox{and every} \ \ j = 0, 1.
\]
 Note that, by  construction, for each $i \geq 0$ and $0 \leq j \leq 1$, the element
$t_{ij} = x^{d_{ij0}} \pi^{d_{ij1}} \in K' \ \ \mbox{has valuation} \ \ \nu_{ij}.$
Since $G_{i'} \subset G_{i}$ for $i' < i$, each $t_{i'j}$ is in a unique way  a monomial in the elements $t_{i0},  t_{i1}$: 
\[
\mbox{if we express} \ \  \nu_{i'j} = b_{i'ij0} \nu_{i0} +  b_{i'ij1} \nu_{i1} \ \ \mbox{with}\ \
 b_{i'ij} \in {\bf Z}_{\geq 0}, \ \ \mbox{then} \ \ t_{i'j} = t_{i0}^{b_{i'ij0}}  t_{i1}^{b_{i'ij1}}.
\]

As the valuations of $t_{i0},  t_{i1}$ are $\bf Z$-linearly independent, the $V$-subalgebra $V[t_{i0},  t_{i1}]$ of $K'$ is  isomorphic with 
\[
V[T_{i0},  T_{i1}]/(\pi - T_{i0}^{b_{i0}}  T_{i1}^{b_{i1}}) \ \ \mbox{with} \ \ b_{ij} = b_{0i1j}  ,
\]
where $\gcd(b_{i0}, b_{i1}) = 1$. In particular, we obtain a nested sequence of $V$-subalgebras 
\[
R_i = V[t_{i0},  t_{i1}]_{(\mm,t_{i0},\,  t_{i1})} \subset V'
\]
that are injectively ind-ci over $V$.

For $d\in R_i$ with $\val(d)\in G$, $a\in \qq'$ and $b\in V$ with $b\not \in \qq \cap V$ set $y_{ad}=a/d$, $z_{b,d}=d/b$ and  consider  $B_{i,d,a,b}=R_i[y_{ad},z_{b,d}]\cong R_i[Y,Z]/(YZ-a/b)$. Then the union $A$ of all 
$B_{i,d,a,b}$ is filtered and injectively ind-ci over $V$. We have to show that $A_{\mm'\cap A}=V'$.

As in the proof of Lemma \ref{g}, it remains to show that  all $(c_e/c_s)x^{e-s}$, $c_i,c_s\in V$  from $\mm'$ are in fact in $A$. 
If $(c_e/c_s)x^{e-s}= dt$ for some $d\in R_i$ with $\val(d)\in G$, $i\in {\bf N}$ and $t\in K$ with $t\not \in \qq\cap V$
then either  $(c_e/c_s)x^{e-s}\in V[d]\subset R_i$ if $\val(t)\geq 0$, or $(c_e/c_s)x^{e-s}\in V[d/b]\subset B_{i,d,a,b} \subset A$ for any $a$ and  $b=t^{-1}$ in case $\val(t)<0$. 

Now assume that $(c_e/c_s)x^{e-s}= dt$ for some $d\in K'$ with $\val(d)\in G$ and $t\in \qq'$. If $\val(d)\geq 0$ then  $(c_e/c_s)x^{e-s}\in V[d]\subset R_i$ for some $i$. If $\val(d)< 0$ then $d^{-1}\in R_i $ for some   
   $i$ and  $(c_e/c_s)x^{e-s}\in V[a/d^{-1}]\subset B_{i,d^{-1},a,b}$
for $a=t$ and any $b$. Thus in both cases $(c_e/c_s)x^{e-s}\in A$.
\hfill\ \end{proof}

\begin{Proposition} \label{pr2}
Let $V\subset V'$ be an extension of valuation rings with  trivial residue field extension,  $K\subset K'$ their fraction field extension, $\Gamma\subset \Gamma'$ their value group extension and $(x_i)_{i\in I}$ a set of elements of $V'$ inducing a ${\bf Z}$-basis of $\Gamma'/\Gamma$. Assume  that  $K'=K((x_i)_{i\in I})$.
Then $V'$ is  injectively ind-ci
over $V$.
\end{Proposition}
\begin{proof} If $I$ is finite then apply Lemma \ref{g} and Proposition \ref{p9} by recurrence. If $I$ is infinite note that $V'$ is a filtered union of $V'\cap K((x_i)_{i\in J})$, where $J$ runs in the set of all finite subsets of $I$.
\hfill\ \end{proof}

\begin{Proposition} \label{p3} Let  $ V\subset V'$ be an   extension  of  valuation rings with  trivial residue field extension,   $\Gamma\subset \Gamma'$ its value group extension, $K\subset K'$ their fraction field extension, $k$ the residue field of $V$,  and $p=\ch(k)$.
  Assume that 
 \begin{enumerate}
  \item $V'$ is Henselian,
 \item    $\Gamma'/\Gamma$ is a finitely   generated group without $p$-torsion if $p>0$.
 \end{enumerate}
  Then $V'$ is   ind-ci over $V$.
 \end{Proposition}
\begin{proof} Let $\Gamma_1\subset \Gamma'$ be the subgroup containing $\Gamma$ which defines the torsion part of $\Gamma'/\Gamma$. Choose $x\in V'$ such that $\val(x)\in \Gamma_1$ and apply Proposition \ref{pr} for $V\subset V_1=V'\cap K(x)$. Then the value group of $V_1$ is $\Gamma+{\bf Z}\val(x)$. Applying step by step Proposition \ref{pr} we arrive to a valuation subring $V''\subset V'$ with the value group $\Gamma''$ such that $\Gamma'/\Gamma''$ is torsion free and $V''$ is  injectively ind-ci over $V$. Then $\Gamma'/\Gamma$ is free being finitely generated torsion free and   applying 
Proposition \ref{pr2} we get  a valuation subring $V_2\subset V'$ such that the extension  $V_2\subset V'$ is immediate and $V_2$ is  injectively  ind-ci over $V''$. Now it is enough to apply Theorem 
\ref{T2} for $V_2\subset V'$.
 \hfill\ \end{proof}
 
 \vskip 0.3 cm
 \section{Extensions of valuation rings with finitely generated residue field extensions}
 
A field extension $K\subset K'$ is {\em separably generated} if $K'$ is an algebraic separable extension of a pure transcendental extension of $K$. Thus a separable finite type field extension is separably generated. 
 
\begin{Lemma}\label{L}   Let  $ V\subset V'$ be an   extension  of  valuation rings and $k\subset k'$ its residue field extension. Assume  $V'$ is Henselian and $k'/k$ is separably generated. Then $V'$ is an extension of a valuation ring $W$ containing $V$ such that 
\begin{enumerate}
\item  the    value group extension of  $V\subset W$ is trivial, 

\item the residue field extension of $W\subset V'$ is trivial, and
\item $W$ is a filtered union of its smooth $V$-subalgebras.

\end{enumerate}
\end{Lemma}
\begin{proof}  Let $\mm,\mm'$ be the maximal ideals of $V, V'$.
By hypothesis there exists a system of elements $x$ of $V'$ inducing a separable transcendental basis of $k'$ over $k$. Then $V''=V[x]_{\mm V[x]}\cong V[X]_{\mm V[X]}$ is a valuation ring with the same value group as $V$ and   the residue field extension of $V''\subset V'$ is algebraic separable. Let ${\bar y}\in k'$ which is not in the residue field $k''$ of $V''$ and ${\bar f}=$\ Irr$({\bar y},k'')\in k''[Y]$. Let $f\in V''[Y]$ be a monic polynomial lifting $\bar f$. As $V'$ is Henselian we may lift $\bar y$ to a solution of $f$ in $V'$ and $V_1=(V''[Y]/(f))_{\mm V''[Y]}$ is a valuation ring with the same value group as $V$ and such that $\bar y$ is contained in the residue field of $V_1$. Using this trick by transfinite induction or by Zorn's Lemma we get the valuation ring $W$.   
\hfill\ \end{proof}

\begin{Theorem} \label{t} Let  $ V\subset V'$ be an   extension  of  valuation rings,   $\Gamma\subset \Gamma'$ its value group extension, $K\subset K'$ their fraction field extension, $k\subset k'$ its residue field extension,  and $p=\ch(k)$.
  Assume that 
 \begin{enumerate}
  \item $V'$ is Henselian and $k'/k$ is separably generated,
 \item    $\Gamma'/\Gamma$ is a finitely generated group without $p$-torsion if $p>0$.
 \end{enumerate}
  Then $V'$ is    ind-ci over $V$.
 \end{Theorem}
 \begin{proof} Applying Lemma \ref{L} we  see that $V'$ is an extension of a valuation subring $W$ such that the residue field extension of  $ W\subset V'$ is trivial, $W$ being a filtered union of smooth $V$-algebras (etale neighborhoods of some polynomial $V$-algebras, see \cite[Theorem 2.5]{S}). Using Proposition \ref{p3} for $W\subset V'$ we are done.
\hfill\ \end{proof}

\vskip 0.5 cm

\section{The proof of Theorem \ref{T4}}

In this section we use methods from Model Theory and
we need mainly \cite[Proposition A.6]{P},  which is obtained using \cite[Theorem 6.1.4]{CK} and says in particular the following:

\begin{Proposition}\label{kes} Let $V$ be a valuation ring with value group $\Gamma$. Then there exists an ultrafilter ${\sU}$ on a set $U$ such that any system of polynomial equations
$(g_i((X_j)_{j\in J})_{i\in I}$ with $\card(I)\leq \card(U)$ in variables $(X_j)_{j\in J}$ with coefficients in  the ultrapower ${\tilde V}=\Pi_{{\sU}}V$ has a solution in  ${\tilde V}$ if and only if all its 
finite subsystems have.
\end{Proposition}

Next theorem extends Theorem \ref{t} and this is the goal of this section.

\begin{Theorem} \label{t9} Let  $ V\subset V'$ be an   extension  of  valuation rings,   $\Gamma\subset \Gamma'$ its value group extension, $K\subset K'$ their fraction field extension, $k\subset k'$ its residue field extension,  and $p=\ch(k)$.
  Assume that 
 \begin{enumerate}
  \item $V'$ is Henselian.
\item $k'/k$ is separable,
 \item    $\Gamma'/\Gamma$   has no $p$-torsion, if $p>0$.
 \end{enumerate}
  Then $V'$ is  ind-ci filtered  over $V$. 
\end{Theorem}
\begin{proof} We apply the  transfinite induction to find a valuation subring $W$ of $V'$, which is ind-ci over $V$ such that for its value group $L$ we have $\Gamma'/L$ torsion free. Set $V_0=V$. For a succesor ordinal  number $\lambda$ choose an element $\gamma$ in $\Gamma'\setminus \Gamma_{\lambda-1}$,  where $\Gamma_{\lambda-1}$ is the value group of $V_{\lambda-1}$ already defined, which induces a torsion element in $\Gamma'/\Gamma$.  As in Proposition \ref{pr} we see that there exist a valuation ring $V_{\lambda}$ such that its value group $\Gamma_{\lambda}$ contains $\gamma$. If $\lambda $ is a limit ordinal number then set  $V_{\lambda}$
for the union of all $V_{\nu}$ for $\nu<\lambda$. The induction stops with some $W=V_{\lambda}$ when there exists no torsion element $\gamma$ in $\Gamma'/\Gamma_{\lambda}$. Clearly, all  $V_{\lambda}$ are injectively ind-ci over $V$.  Changing $V$ by $W$ we may suppose that $\Gamma'/\Gamma$ is torsion free.

Let  $\Gamma_1\subset \Gamma'$  be a subgroup containing $\Gamma$ such that $\Gamma_1/\Gamma$  is finitely generated and $k_1\subset k'$ a  subfield containing $k$ such that $k_1/k$ is  finitely generated. 

 We can find a valuation subring $V_{k_1,\Gamma_1}\subset V'$  containing $V$ with $k_1$ its residue field, $\Gamma_1$ its value group and such that $V_{k_1,\Gamma_1}$
is injectively ind-ci over $V$  (note that a separable finitely generated field  extension is separably generated). Indeed, the proofs from the previous sections use mostly  that $V'$ is Henselian.

 Note that $\Gamma_1/\Gamma$ is free, let us say $\val(x_1)$ for some $x_1 $ in $V'$, is a basis of $\Gamma_1/\Gamma$. Then $\Gamma_1$ is the value group of $V'\cap K(x_1)$. Similarly, we may consider some elements $x_2$ in $V'$, which induces a transcendental basis $\bar x_2$ of $k'/k$ such that $k'/k(\bar x_2)$ is algebraic separable. Note that $x_1,x_2$ form an algebraically independent system of elements over $K$.  Set $W_2=V'\cap K(x_1,x_2)$. Choose an element $z$ in $V'$, which induces a primitive  algebraic separable element of
$k_1/k({\bar x}_2)$. Set $V_{k_1,\Gamma_1}=V'\cap K(x_1,x_2,z)$. Clearly , the residue field of $V_{k_1,\Gamma_1}$ is $k_1$ and its value group is $\Gamma_1$. Note that $V_{k_1,\Gamma_1}$ is essentially generated by $x_1,x_2,z$ and the kernel of the map 
$V[X_1,X_2,Z]\to V' $, $(X_1,X_2,Z)\mapsto (x_1,x_2,z)$ is generated by a monic separable irreducible polynomial $f_{k_1,\Gamma_1}$ in $Z$ over $V[X_1,X_2]$, $X_1,X_2$ being variables.
Any solution of  $f_{k_1,\Gamma_1}$  in $V'$ defines a valuation subring of $V'$ isomorphic with $V_{k_1,\Gamma_1}$.

 Let $\mathcal{E}$ be the set of all pairs $(k_1,\Gamma_1)$
with $k_1\subset k'$, which is  finitely generated extension of $k$ and $\Gamma_1\subset \Gamma'$ a subgroup containing $\Gamma$ such that $\Gamma'/\Gamma_1$ is finitely generated. For some other  $(k_2,\Gamma_2)\in \mathcal{E}$ with $k_1\subset k_2$, $\Gamma_1\subset \Gamma_2$ maybe $V_{k_1,\Gamma_1}\not \subset V_{k_2,\Gamma_2}$. It is also possible that $V_{k_1,\Gamma_1}$ is not contained in the Henselization $T_{k_2,\Gamma_2}$ of $V_{k_2,\Gamma_2}$, which can be assumed contained in $V'$. However, in the valuation ring $T_{k_2,\Gamma_2}$ there exists a valuation ring of type $V_{k_1,\Gamma_1}$ using the Henselian property.
Similar to the case of $V_{k_1,\Gamma_1}$ we see that $T_{k_1,\Gamma_1}$ is generated by some algebraically independent elements $y'$ over $K$ (they are already present in $V_{k_1,\Gamma_1}$) followed by some integral elements $(y''_i)_i$ over $V[y']$ because $T_{k_1,\Gamma_1}$ is defined by some etale neighborhoods of $V_{k_1,\Gamma_1}$, which are essentially finite over 
$V_{k_1,\Gamma_1}$ (see \cite[Theorem 2.5]{S}). Note that the extension $V_{k_1,\Gamma_1}\subset T_{k_1,\Gamma_1}$ is immediate.

Let $G_{k_1,\Gamma_1}$ be a  system (not finite) of polynomials
over $V$ in some variables  $Y',Y''$, which generates the kernel of the map $V[ Y',Y'']\to V'$, $(Y',Y'')\mapsto  (y',y'')$. A solution of $G_{k_1,\Gamma_1}$ in $V'$ defines a Henselian   valuation subring of $V'$
isomorphic with $T_{k_1,\Gamma_1}$.

For some   $(k_1,\Gamma_1),(k_2,\Gamma_2)\in \mathcal{E}$ with $k_1\subset k_2$, $\Gamma_1\subset \Gamma_2$ assume that $T_{k_1,\Gamma_1} \subset T_{k_2,\Gamma_2}$.  Then there exist some polynomials $H_{k_1,k_2,\Gamma_1,\Gamma_2,i}$ over $V$ such that 
 $$u_{k_1,k_2,\Gamma_1,\Gamma_2,i} y_{k_1,\Gamma_1,i}=H_{k_1,k_2,\Gamma_1,\Gamma_2,i}(y_{k_2,\Gamma_2}),$$
  for some units $u_{k_1,k_2,\Gamma_1,\Gamma_2,i} $ of $T_{k_2,\Gamma_2}$. Let $F_{k_1,k_2,\Gamma_1,\Gamma_2,i}$ be the system of polynomials 
 $$U_{k_1,k_2,\Gamma_1,\Gamma_2,i} Y_{k_1,\Gamma_1,i}-H_{k_1,k_2,\Gamma_1,\Gamma_2,i}(Y_{k_2,\Gamma_2}),$$ 
  and $U_{k_1,k_2,\Gamma_1,\Gamma_2,i}U'_{k_1,k_2,\Gamma_1,\Gamma_2,i}-1$,
  in some variables $ Y_{k_1,\Gamma_1,i}$, $ Y_{k_2,\Gamma_2,j}$, $U_{k_1,k_2,\Gamma_1,\Gamma_2,i}$,\\
  $U'_{k_1,k_2,\Gamma_1,\Gamma_2,i}$.
  
  A solution of the (not finite) system of polynomials 
$G_{k_1,\Gamma_1}$, $G_{k_2,\Gamma_2}$, $(F_{k_1,k_2,\Gamma_1,\Gamma_2,i})_i$ in $V'$  defines an extension of Henselian valuation subrings of type  $T_{k_1,\Gamma_1}\subset T_{k_2,\Gamma_2}$. A solution of all $G_{k_1,\Gamma_1}$,  $(F_{k_1,k_2,\Gamma_1,\Gamma_2,i})_i$ in $V'$  defines a filtered set by inclusion of  Henselian valuation subrings corresponding to $(k_1,\Gamma_1)\in \mathcal{E}$.

  We apply Proposition \ref{kes}. There exists an ultrafilter $\mathcal{P}_1$ on a set 
  $P_1$ with $\card(P_1)$ greater  than the cardinal of the set of all polynomials $G $ and $F$
  with the property  that $G $, $F$   have  a solution in the ultraproduct $V'_1=\Pi_{\mathcal{P}_1}  V'$ because each finite subsystem of them has a solution in $V'_1$. Indeed, for some $(k_1,\Gamma_1),\ldots,(k_r,\Gamma_r)\in \mathcal{E}$ 
 we choose a pair $(k'',\Gamma'')\in \mathcal{E}$ such that $(k_j,\Gamma_j)<(k'',\Gamma'')$, the order being given by inclusion.
   Then we may find  as before  $T_{k_j,\Gamma_j}$ as some valuation  subrings of   $T_{k'',\Gamma''}$. Thus any finite subsystem of $F,G$ defined by $(T_{k_j,\Gamma_j})_j$  has a solution in  $T_{k'',\Gamma''}$ and so in $V'\subset V'_1 $.
  
  Fix such a solution of $G,F$ in $V'_1$. Thus there exists a filtered set by inclusion of Henselian valuation subrings of $V'_1$, which are injectively ind-ci over $V$  and whose union $A_1$ is a valuation ring with residue field $k'$ and value group $\Gamma'$.
  
   Repeating this procedure with $V'_1$ instead $V$ (note that $V'_1$ is still Henselian) we find a set $P_2$ and an ultrafilter $\mathcal{P}_2$ such that $V'_2=\Pi_{\mathcal{P}_2} V'_1$ contains  a valuation ring $A_2$  with the  residue field and value group   of $V'_1$, which is   injectively ind-ci over  $V'_1$. Repeating again this procedure we find
some sets $(P_n)_n$ and some   ultrafilters $(\mathcal{P}_n)_n$ on them and define $V'_{n+1}=\Pi_{\mathcal{P}_{n+1}} V'_n$ and  $\wt V'=\varinjlim_{n \geq 0} V'_n$. In this way we obtain a filtered set  ordered  by inclusion $A_{n+1}\subset V'_{n+1}$ of   valuation rings      with the  residue field  and value group  of $V'_n$, which is injectively ind-ci over  $V'_n$.
 So the union $\wt A$ of $A_n$ is  injectively ind-ci over $V$, has   the same residue field and the same value group with $\wt V'$.

 Note that the extension $\wt A\subset  \wt{V'}$ is immediate. 
  By Theorem \ref{T2} $\wt{V'}$ is  ind-ci over  $\wt A$  and so  ind-ci over $V$. 
Let $E$ be  a finitely presented $V$-algebra and $w:E\to V'$ a $V$-morphism. Then the composite map $E\to V\to \wt{V'}$ factors through a complete intersection $V$-algebra  $D$.  Thus  $w$ 
 factors through $D$ too    because all finite systems of polynomial equations, which have a solution in $\wt V'$ must have one in $V'$. This is enough  by \cite[ Lemma 1.5]{S}.
\hfill\ \end{proof}

\end{document}